\documentclass{amsart}
\usepackage{amssymb}
\usepackage{mathrsfs}
\usepackage{amsmath}
\usepackage{amscd}
\usepackage{mathdots}
\usepackage{graphicx}
\usepackage{subcaption}
\usepackage{wrapfig}
\usepackage{caption}

\usepackage{pinlabel}
\usepackage{enumitem}
\usepackage{graphicx} 

\newtheorem{theorem}{Theorem}
\newtheorem{prop}[theorem]{Proposition}
\newtheorem{lemma}[theorem]{Lemma}
\newtheorem{corollary}[theorem]{Corollary}
\newtheorem{definition}[theorem]{Definition}
\newtheorem{ass}[theorem]{Assumption}
\newtheorem{remark}[theorem]{Remark}

\title{Harmonic metrics of generically regular nilpotent Higgs bundles over non-compact surfaces}

\numberwithin{equation}{section}
\numberwithin{theorem}{section}

\usepackage{fancyhdr}

\author{Song Dai}

\address{Song Dai\\Center for Applied Mathematics and KL-AAGDM, Tianjin University\\No. 92 Weijinlu Nankai District\\ Tianjin\\ P.R. China 300072}
\email{song.dai@tju.edu.cn}

\author{Qiongling Li}

\address{Qiongling Li\\
Chern Institute of Mathematics and LPMC\\
Nankai University\\
No. 94 Weijinlu Nankai District\\
Tianjin\\
P.R.China 300071}

\email{qiongling.li@nankai.edu.cn}

\begin{document}
\pagestyle{fancy}
\fancyhead{} 
\fancyhead[CO]{Harmonic metrics of generically regular nilpotent Higgs bundles}
\fancyhead[CE]{Song Dai and Qiongling Li}
\begin{abstract}

A rank $n$ Higgs bundle $(E,\theta)$ is called generically regular nilpotent if $\theta^n=0$ but $\theta^{n-1}\neq 0$. 
We show that for a generically regular nilpotent Higgs bundle, if it admits a harmonic metric, then its graded Higgs bundle admits a unique maximal harmonic metric. The proof relies on a generalization of Kalka-Yang's theorem for prescribed curvature equation over a non-compact hyperbolic surface to a coupled system. As an application, we show that the branched set of a branched minimal disk in $\mathbb{H}^3$ has to be the critical set of some holomorphic self-map of $\mathbb{D}$. 
\end{abstract}

\maketitle


\section{Introduction}

Let \((E, \theta)\) be a Higgs bundle over a Riemann surface \(X\). A Hermitian metric \(h\) on \(E\) is called harmonic if it satisfies the Hitchin equation
\[
F(\nabla_h) + [\theta, \theta^{*_h}] = 0.
\]
For compact Riemann surfaces \(X\), it is well-known from the works of Hitchin \cite{selfduality} and Simpson \cite{s1} that \((E, \theta)\) admits a harmonic metric if and only if it is polystable of degree \(0\).

In the case of non-compact \(X\), the criteria for the existence of harmonic metrics are less well understood. For example, when \(X = \bar{X} \setminus D\) and \((E, \theta)\) is wild on \((\bar{X}, D)\), results by Simpson \cite{SimpsonNoncompact}, Biquard-Boalch \cite{Biquard-Boalch}, and Mochizuki \cite{Mochizuki-wild} indicate that \((E, \theta)\) admits a harmonic metric if and only if it can be extended to a polystable good filtered Higgs bundle over \(\bar{X}\). However, even simpler cases, such as when \(X = \mathbb{D}\), fall outside the framework of wild Higgs bundles.

Recent work by T. Mochizuki and the second author has addressed more general scenarios beyond the wild Higgs bundle framework. They have developed techniques to establish the existence of harmonic metrics without stability conditions. In \cite{LiMochizukiGeneric}, they showed that for a \(GL(n, \mathbb{R})\)-Higgs bundle over any non-compact Riemann surface, if the Higgs field is generically regular semisimple, then a harmonic metric compatible with the \(GL(n, \mathbb{R})\)-structure exists. Furthermore, in \cite{LiMochizukiHitchinSection}, they proved that if \(X\) is hyperbolic, a Higgs bundle in the $SL(n,\mathbb R)$-Hitchin section admits a harmonic metric. See also the work of \cite{fujioka2024harmonic} and \cite{ma2024}.

In this paper, we investigate the existence of harmonic metrics for nilpotent Higgs bundles \((E, \theta)\) over non-compact Riemann surfaces \(X\). Notably, nilpotent Higgs bundles do not admit harmonic metrics over \(\mathbb{C}\) or \(\mathbb{C}^*\). Therefore, our focus is on the case where \(X\) is Poincaré hyperbolic, meaning that \(X\) admits a complete hyperbolic metric \(g_X\). According to \cite{LiNilpotent}, for any harmonic metric \(h\) on a nilpotent Higgs bundle \((E, \theta)\) over \(X\), the norm of \(\theta\) with respect to \(h\) and \(g_X\) satisfies
\[
||\theta||_{h,g_X}^2 \leq \frac{n(n^2-1)}{12}.
\]

Note that any nilpotent Higgs bundle together with a harmonic metric $h$ gives rise to an equivariant minimal immersion of $\widetilde{X}$ in the symmetric space $GL(n,\mathbb C)/U(n)$ with the left $GL(n,\mathbb C)$-invariant Riemannian metric induced by the trace form on $\mathfrak gl(n,\mathbb C)$. The induced Hermitian metric is $|\theta|_h^2$, and thus $\leq \frac{n(n^2-1)}{12}\hat g_X.$

We establish the following results:

\begin{theorem}\label{intro:main1}
    (Theorem \ref{main theoremLater})
    Let \((E, \theta)\) be a generically regular nilpotent Higgs bundle over a hyperbolic Riemann surface \(X\). If a harmonic metric \(h_0\) exists on \((E, \theta)\), then there exists a metric \(h\) on the graded Higgs bundle \(\text{Gr}_{\mathcal{G}}(E, \theta) = (\oplus_{i=1}^n \text{Gr}_iE, \oplus_{i=1}^{n-1} \text{Gr}_i \theta)\) such that \(\det(h) = \det(h_0)\). 
\end{theorem}

\begin{theorem}\label{intro:main2}
    (Theorem \ref{main theoremLater})
    Let $(E,\theta)$ be a CVHS of type $(1,\cdots,1)$ over $X$ of the form:
        \[
    E = \oplus_{i=1}^n L_i, \quad \theta = \begin{pmatrix}
0 & & & \\
\gamma_1 & 0 & & \\
& \ddots & \ddots & \\
& & \gamma_{n-1} & 0
\end{pmatrix},
    \] where $L_i(i=1,\cdots,n)$ are holomorphic line bundles and $\gamma_i\in H^0(X,\text{Hom}(L_i,L_{i+1})\otimes K_X)$ is nonzero for $i=1,\cdots,n-1$.
    
    If there exists a harmonic metric \(h_0\) exists on \((E, \theta)\), then there uniquely exists a maximal harmonic metric \(h_{max}\) on $(E,\theta)$ satisfying \(\det(h_{max}) = \det(h_0)\). Furthermore, \(h_{max}\) is diagonal, that is, $h_{max}=\oplus_{i=1}^nh_{max}|_{L_i}$.
\end{theorem}

\begin{theorem}(Theorem \ref{main theoremLater1})
    For the base \(n\)-Fuchsian Higgs bundle, where
    \[
    E = \oplus_{i=1}^n K_X^{\frac{n+1-2i}{2}}, \quad \theta = \begin{pmatrix}
0 & & & \\
\gamma_1 & 0 & & \\
& \ddots & \ddots & \\
& & \gamma_{n-1} & 0
\end{pmatrix},\quad \text{for $\gamma_i=\sqrt{\frac{i(n-i)}{2}},$}
    \] 
    the Hermitian metric \(h_X = \oplus_{i=1}^n \hat g_X^{-\frac{n+1-2i}{2}}\) is the maximal harmonic metric among all (not necessarily diagonal) harmonic metrics of unit determinant.
\end{theorem}

Here, "generically regular nilpotent" means that \(\theta^n = 0\) and \(\theta^{n-1} \neq 0\). The graded bundle \(\text{Gr}_{\mathcal{G}}(E, \theta)\) is constructed as follows. Let \(G_i\) be the subbundle generated by the subsheaf \(\ker \theta^{n-i}\). Then
\[
\mathcal{G} = \{ E = G_0 \supset G_1 \supset \cdots \supset G_n = 0 \}
\]
is a full decreasing holomorphic filtration of \(E\), and \(\theta\) maps \(G_{i-1}\) to \(G_i \otimes K_X\). The filtration $\mathcal G$ is called the \textit{canonical filtration} of $(E,\theta).$

The graded Higgs bundle \(\text{Gr}_{\mathcal{G}}(E, \theta)\) is constructed as:
\[
(E^0 = \bigoplus_{i=1}^n G_{i-1}/G_i, \quad \theta^0 = \sum_{i=1}^n \text{Gr}_i \theta),
\]
where for \(i = 1, \ldots, n-1\),
\[
\text{Gr}_i \theta : G_{i-1}/G_i \to G_i/G_{i+1} \otimes K_X
\]
is induced from \(\theta\), and \(\text{Gr}_n \theta : G_{n-1} \to 0\). The graded Higgs bundle \(\text{Gr}_{\mathcal{G}}(E, \theta)\) is a CVHS of type $(1,\cdots,1)$ and can be viewed as the limit as \(t \to \infty\) of the \(\mathbb{C}^*\)-flow.

We define a partial order on Hermitian metrics as follows:

\begin{definition}
    Let \((E, \theta)\) be a generically regular nilpotent Higgs bundle with the canonical filtration \(\mathcal G=\{G_k\}\). Two Hermitian metrics \(h_1\) and \(h_2\) on \(E\) are said to have \(h_1\) weakly dominating \(h_2\) if \(\det(h_1|_{G_k}) \geq \det(h_2|_{G_k})\) for \(k = 0, \ldots, n-1\).
\end{definition}

A harmonic metric is maximal if it weakly dominates all other harmonic metrics. While a maximal harmonic metric may not always exist, if it does, it is unique.

The proof of Theorem \ref{intro:main1} proceeds in two steps. First, we show that a solution to the Hitchin equation on a generically regular Higgs bundle induces a supersolution to the equation on \(\det G_k\)'s induced from Hitchin equation, as detailed in Proposition \ref{prop 1}. Next, we show that if a supersolution to the equation on \(\det G_k\)'s exists, then a maximal solution to the equation on \(\det G_k\)'s also exists. This maximal solution then provides a maximal harmonic metric on \(\text{Gr}_{\mathcal{G}}(E, \theta)\). We finish the second step by showing a more general result in partial differential equations, Proposition \ref{prop:introductionmain1} below.

\begin{prop}\label{prop:introductionmain1}
    (Proposition \ref{prop:main1})
    Let \((X, g_X)\) be a non-compact hyperbolic Riemann surface. Let \(k_i\), \(i = 1, \ldots, n-1\), be essentially positive smooth functions on \(X\). Consider the system
    \begin{equation}\label{intro:system}
    \Delta_{g_X} u_i + i(n-i) - k_i e^{2u_i - u_{i-1} - u_{i+1}} = 0, \quad 1 \leq i \leq n-1, \quad u_0 = u_n = 0.
    \end{equation}
    If this system has a subsolution, then it has a maximal solution.
\end{prop}

When \(n = 2\), Equation \eqref{intro:system} simplifies to
\begin{equation}\label{eqn:introsingle}
    \Delta_{g_X} u + 1 - k e^{2u} = 0.
\end{equation}
This is the prescribed curvature equation, which seeks a function \(u\) such that the metric \(e^{2u} \cdot g_X\) has curvature \(-k\). In the special case where \(k \equiv 1\), the maximal solution is \(u = 0\), which represents the complete hyperbolic metric \(g_X\). Here, "essentially positive functions" are defined in Definition \ref{essentially positive}. Proposition \ref{prop:introductionmain1} extends the following theorem by Kalka and Yang for a single scalar equation. 

\begin{theorem}(Kalka-Yang \cite[Theorem 4]{KalkaYang})
Let \(k\) be an essentially positive smooth function on \(X\). If \eqref{eqn:introsingle} has a subsolution, then it has a maximal solution.
\end{theorem}

Theorem \ref{intro:main1} provides a necessary condition for the existence of harmonic metrics on generically regular nilpotent Higgs bundles: the existence of harmonic metric on its graded Higgs bundle, a CVHS of type $(1,\cdots,1)$.For the existence of a harmonic metric on a CVHS of type $(1,\cdots,1)$ parametrized by $(\gamma_1,\cdots, \gamma_{n-1})$, we have the following facts in Section \ref{sec:non-existence}:
\begin{itemize}
\item The existence property is preserved under holomorphic covering map. (Lemma \ref{lem:ExistenceLift}). 
\item Over $\mathbb D$, if $\gamma_i$'s are bounded, then a harmonic metric exists. (Lemma \ref{lem:ExistenceCVHS}).
\item Harmonic metrics may not exist on some $(E,\theta)$.(Proposition \ref{prop:NonExistence}).
\end{itemize}

Additionally, Theorem \ref{intro:main1} has a geometric application. For \(n = 2\), it is related to branched minimal surfaces in \(\mathbb{H}^3\). The branched points of these surfaces correspond to the critical points of a holomorphic map from \(\widetilde{X}\) to \(\mathbb{D}\), where \(\widetilde{X}\) is the universal cover of \(X\).

\begin{theorem}[Theorem \ref{branched minimal}]
    Let \(f: \widetilde{X} \to \mathbb{H}^3\) be an equivariant branched minimal disk which does not lie in a totally geodesic copy of $\mathbb D$. Then it induces an equivariant holomorphic map \(f^0: \widetilde{X} \to \mathbb{D}\) satisfying \[0<\frac{g_f}{g_{f_0}}<1.\] Furthermore, the branched points of \(f\) coincide with the critical points of \(f^0\).
\end{theorem}

Kraus \cite{Kraus} has characterized the set of critical points of a holomorphic function from \(\mathbb{D}\) to itself. For further details, see Section \ref{n=2}.
\vspace{0.2cm}

\textbf{Organization:} In Section \ref{generically regular}, we provide some basic results on generically regular nilpotent Higgs bundles. In Section \ref{harmonic metrics and main theorem}, we show that a harmonic metric on $(E,\theta)$ induces a supersolution to the equation on $\det G_k$'s. In Section \ref{Kalka-Yang}, we prove Proposition \ref{prop:main1} as a generalization of Kalka-Yang's theorem. In Section \ref{sec:ProofofMainTheorem}, we prove the main theorem: Theorem \ref{main theoremLater}. In Section \ref{sec:non-existence}, we discuss the existence of harmonic metrics on a CVHS of type $(1,\cdots,1)$. In Section \ref{further discussion}, we further consider the related topics on $\mathbb{C}^*$-flow and minimal branched disk in $\mathbb{H}^3.$

\vspace{0.2cm}
\textbf{Acknowledgement:} Both authors thank the hospitality of Max Planck Institute at Leipzig, where part of the work was completed. S. Dai is supported by NSF of China No.12071338. Q. Li is partially supported by the National Key R\&D Program of China No. 2022YFA1006600, the Fundamental Research
Funds for the Central Universities, and Nankai Zhide foundation. Q. Li is also sponsored by the Alexander von Humboldt Foundation.


\section{Preliminaries on generically regular nilpotent Higgs bundles}\label{generically regular}
Let \(X\) be a non-compact hyperbolic Riemann surface with canonical line bundle \(K_X\), and let \(g_X\) be the complete hyperbolic Kähler metric on \(X\). Consider a Higgs bundle \((E, \theta)\) over \(X\) of rank \(n\).

\begin{definition}
A Higgs bundle \((E, \theta)\) is called generically regular nilpotent if \(\theta^n = 0\) and \(\theta^{n-1} \neq 0\).
\end{definition}

\begin{definition}\label{complex variation of Hodge structure}
A Higgs bundle \((E, \theta)\) is called a complex variation of Hodge structure of type \((1, \cdots, 1)\) if it can be expressed as
\[ E = L_1 \oplus L_2 \oplus \cdots \oplus L_n, \]
where \(L_i\) are holomorphic line bundles for \(i = 1, \cdots, n\), and the Higgs field \(\theta\) is given by
\[ \gamma_i = \theta|_{L_i}: L_i \to L_{i+1} \otimes K_X \text{ for } i = 1, \cdots, n-1, \]
where each \(\gamma_i\) is a non-zero holomorphic bundle map, and \(\gamma_n = \theta|_{L_n}: L_n \to 0\). We write \(\theta = \sum_{i=1}^n \gamma_i\).
\end{definition}

\begin{remark}
A complex variation of Hodge structure of type \((1, \cdots, 1)\) is generically regular nilpotent.
\end{remark}

For a generically regular nilpotent Higgs bundle \((E, \theta)\), there is a canonical full decreasing holomorphic filtration \(\mathcal{G}\). Let \(G_i\) be the subbundle given by the subsheaf \(\ker \theta^{n-i}\). Then
\[ \mathcal{G} = \{ E = G_0 \supset G_1 \supset \cdots \supset G_n = 0 \} \]
is a full decreasing holomorphic filtration of \(E\), and \(\theta\) maps \(G_{i-1}\) to \(G_i \otimes K_X\).

Given \(\mathcal{G}\), we can construct the graded Higgs bundle \(\text{Gr}_{\mathcal{G}}(E, \theta)\) as follows:
\[ (E^0 = \bigoplus_{i=1}^n G_{i-1}/G_i, \quad \theta^0 = \sum_{i=1}^n \text{Gr}_i \theta), \]
where for \(i = 1, \cdots, n-1\),
\[ \text{Gr}_i \theta: G_{i-1}/G_i \to G_i/G_{i+1} \otimes K_X \]
is induced from \(\theta\), and \(\text{Gr}_n \theta: G_{n-1} \to 0\). Since \((E, \theta)\) is generically regular, \(\text{Gr}_i \theta\) is not identically zero for \(i = 1, \cdots, n-1\). Thus, \(\text{Gr}_{\mathcal{G}}(E, \theta)\) is a complex variation of Hodge structure of type \((1, \cdots, 1)\). Setting \(L_i = G_{i-1}/G_i\) and \(\gamma_i = \text{Gr}_i \theta\), the graded Higgs bundle \(\text{Gr}_{\mathcal{G}}(E, \theta)\) can be expressed as
\begin{equation}\label{graded bundle}
(E^0, \bar{\partial}_{E^0}) = (L_1, \bar{\partial}_1) \oplus \cdots \oplus (L_n, \bar{\partial}_n), \quad
\theta^0 = \begin{pmatrix}
0 & & & \\
\gamma_1 & 0 & & \\
& \ddots & \ddots & \\
& & \gamma_{n-1} & 0
\end{pmatrix}.
\end{equation}

If \((E, \theta)\) is a complex variation of Hodge structure of type \((1, \cdots, 1)\), then \(\mathcal{G}\) is its canonical decreasing filtration:
\[ G_i = \bigoplus_{k \geq i+1} L_k. \]
Thus, \(G_{i-1}/G_i \cong L_i\) canonically and
\[ \text{Gr}_{\mathcal{G}}(E, \theta) \cong (E, \theta). \]

For a generically regular nilpotent Higgs bundle \((E, \theta)\) with filtration \(\mathcal{G} = \{ G_i \}\), the graded Higgs bundle \(\text{Gr}_{\mathcal{G}}(E, \theta) = (E^0, \theta^0)\) has a canonical filtration \(\mathcal{G}^0 = \{ G_i^0 \}\), where \(G_i^0 = \bigoplus_{k \geq i+1} G_{k-1}/G_k\). Although \(G_i\) and \(G_i^0\) are generally not isomorphic, the determinant line bundles \(\det(G_i)\) and \(\det(G_i^0)\) are canonically isomorphic for \(i = 1, \cdots, n-1\).

We now consider Higgs bundles with Hermitian metrics.

\subsection{Smooth decomposition of $E$}\label{subsection:smoothdecom}
Let \((E, \theta)\) be a generically regular nilpotent Higgs bundle with \(G_i = \ker \theta^{n-i}\) for \(0 \leq i \leq n\). Let \(h\) be a Hermitian metric on \(E\). Denote \(L_i\) as the \(h\)-orthogonal complement of \(G_i\) in \(G_{i-1}\) for \(1 \leq i \leq n\). Using \(h\), \(L_i \cong G_{i-1}/G_i\) canonically, and thus we have a smooth decomposition
\begin{equation}
E=L_1 \oplus L_2 \oplus \cdots \oplus L_n. 
\end{equation}
With respect to this decomposition, we have the following:

I. The Hermitian metric \(h\) is given by
\begin{equation}\label{Metric}
h = \begin{pmatrix}
h_1 & & & \\
& h_2 & & \\
& & \ddots & \\
& & & h_n
\end{pmatrix},
\end{equation}
where \(h_i\) is the restriction of \(h\) to \(L_i\);

II. The holomorphic structure on \(E\) is defined by the \(\bar{\partial}\)-operator
\begin{eqnarray}\label{cpx str}
\bar{\partial}_E = \begin{pmatrix}
\bar{\partial}_1 & 0 & 0 & \cdots & 0 \\
\beta_{21} & \bar{\partial}_2 & 0 & \cdots & 0 \\
\beta_{31} & \beta_{32} & \bar{\partial}_3 & \cdots & 0 \\
\vdots & \vdots & & \ddots & \vdots \\
\beta_{n1} & \beta_{n2} & \cdots & \beta_{n, n-1} & \bar{\partial}_n
\end{pmatrix},
\end{eqnarray}
where \(\bar{\partial}_k\) is the \(\bar{\partial}\)-operator defining the holomorphic structure on \(G_{k-1}/G_k\), and \(\beta_{ij} \in \Omega^{0,1}(X, \text{Hom}(L_j, L_i))\);

III. The Higgs field is of the form
\begin{eqnarray}\label{Higgs field}
\theta = \begin{pmatrix}
0 & 0 & 0 & \cdots & 0 \\
\gamma_1 & 0 & 0 & \cdots & 0 \\
a_{31} & \gamma_2 & 0 & \cdots & 0 \\
\vdots & \vdots & \ddots & \ddots & \vdots \\
a_{n1} & a_{n2} & \cdots & \gamma_{n-1} & 0
\end{pmatrix},
\end{eqnarray}
where \(\gamma_k: L_k \to L_{k+1} \otimes K_X\) is the holomorphic map \(\text{Gr}_k \theta: G_{k-1}/G_k \to G_k/G_{k+1} \otimes K_X\), and \(a_{ij} \in \Omega^{1,0}(X, \text{Hom}(L_j, L_i))\).

Note that in the expressions above, \(\bar{\partial}_k\) and \(\gamma_k\) are determined by the Higgs bundle \((E, \theta)\) and are independent of the Hermitian metric \(h\). However, \(a_{ij}\) and \(\beta_{ij}\) depend on the Hermitian metric \(h\). The graded Higgs bundle \((E^0, \theta^0)\) in (\ref{graded bundle}) with diagonal Hermitian metric \(h = (h_1, \cdots, h_n)\) is the case where all \(a_{ij}\) and \(\beta_{ij}\) vanish.

Since \(\det(G_k) \cong \det(G_k^0)\), we can consider the induced metric \(\det(h|_{G_k})\) on the determinant line bundle \(\det(G_k)\) or \(\det(G_k^0)\). In general, for a complex vector space \((V, h)\) with a Hermitian metric and a subspace \(W\) of \(V\), we have
\[ \det W \otimes \det (V/W) \cong \det V. \]
Thus, we formally define
\[ \det(h_{V/W}) = \frac{\det(h_V)}{\det(h_W)}. \]

\begin{remark}
In the formula above, \(\det(h_{V/W})\) is defined using the Hermitian metric on the determinant bundles. If we use the Hermitian metric on \(V\), \(h_{V/W}\) can be directly defined by using \(V/W \cong W^\perp\).
\end{remark}

In our case, with diagonal Hermitian metric \(h = \text{diag}(h_1, \cdots, h_n)\), we have \(\det(h|_{G_k}) = \prod_{i=k+1}^n h_i\). Thus,
\begin{equation}\label{fil-line}
h_k = \det(h_{L_k}) = \frac{\det(h|_{G_{k-1}})}{\det(h|_{G_k})}.
\end{equation}

\begin{definition}\label{weakly dominate}
Let \((E, \theta)\) be a generically regular nilpotent Higgs bundle with canonical filtration \(\{ G_k \}\). Let \(h_1\) and \(h_2\) be two Hermitian metrics on \(E\). We say that \(h_1\) weakly dominates \(h_2\) if
\[ \det(h_1|_{G_k}) \geq \det(h_2|_{G_k}) \]
for every \(k = 0, \cdots, n-1\).
\end{definition}

One may extend the definition of weakly dominate between Hermitian metrics on a nilpotent Higgs bundle $(E,\theta)$ and its graded Higgs bundle. 
\begin{definition}\label{weakly dominategen}
Let \((E, \theta)\) be a generically regular nilpotent Higgs bundle with canonical filtration \(\{ G_k \}\). Let \((E^0, \theta^0)\) be its graded Higgs bundle with canonical filtration \(\{ G_k \}\). Let \(h_1\) and \(h_2\) be two Hermitian metrics on \(E\), \(E^0\) respectively. We say that \(h_1\) weakly dominates \(h_2\) if 
\[ \det(h_1|_{G_k}) \geq \det(h_2|_{G_k^0}) \] under the canonical isomorphism $\det(G_k)\cong \det(G_k^0)$
for every \(k = 0, \cdots, n-1\).

We say that \(h_1\) strictly weakly dominates \(h_2\) if under the canonical isomorphism $\det(G_k)\cong \det(G_k^0)$
for every \(k = 0, \cdots, n-1\),
\[ \det(h_1|_{G_0}) \geq \det(h_2|_{G_0^0})\] 
and
\[ \det(h_1|_{G_k}) > \det(h_2|_{G_k^0})\quad \text{for $k=1,\cdots, n-1$.}\] 
\end{definition}

\section{Harmonic metrics and Hitchin equation}\label{harmonic metrics and main theorem}
We now turn our attention to harmonic metrics on a Higgs bundle \((E,\theta)\), which are Hermitian metrics satisfying the Hitchin equation:
\begin{equation}\label{Hit eqn}
    F(\nabla_h) + [\theta, \theta^{*_h}] = 0,
\end{equation}
where \(\nabla_h\) is the Chern connection uniquely determined by \(h\) and the holomorphic structure on \(E\), \(F(\nabla_h)\) is the curvature 2-form of \(\nabla_h\), and \(\theta^{*_h}\) is the adjoint of \(\theta\) with respect to \(h\).

Let \((E,\theta)\) be a generically regular nilpotent Higgs bundle, and let \(h\) be a harmonic metric on \((E,\theta)\). Let \(\mathcal{G}\) be its canonical filtration, and \(E = \oplus_i L_i\) be the \(h\)-orthogonal decomposition as in Section \ref{subsection:smoothdecom}. We also adopt the other notations from Section \ref{subsection:smoothdecom}. Let \(s_i\) be a local holomorphic frame for \(L_i\), where \(i = 1, \ldots, n\). Locally, we have
\begin{equation}\label{eqn:localexpression}
    h_i = h_i(z) s_i^* \otimes \bar{s}_i^*, \quad \gamma_i = \gamma_i(z) s_i^* \otimes s_{i+1} \otimes dz,
\end{equation}
where \(h_i(z)\) and \(\gamma_i(z)\) are local functions. Similarly, we obtain the local expression of \(h\) and \(\theta\) as \(h(z)\) and \(\theta(z)dz\). Then, locally,
\[F(\nabla_h) = \bar{\partial}(h(z)^{-1} \partial h(z)),\quad \theta^{*_h} = h(z)^{-1} \theta(z)^* h(z).\]

\subsubsection{Background Metric on \(X\) and \(E\)}

Let \(g_X\) denote the hyperbolic metric on \(X\), and \(\omega_X\) the corresponding Kähler form. Locally, we have
\[
g_X = g_X(z) \, |dz|^2 \quad \text{and} \quad \omega_X = \frac{\sqrt{-1}}{2} g_X(z) \, dz \wedge d\bar{z}.
\]
 So we have $|\partial/\partial_z|_{g_X}^2=\frac{1}{2}g_X(z)$ and $|dz|_{g_X}^2=2g_X(z)^{-1}.$ The induced Hermitian metric on $K_X^{-1}$ is $\hat g_X=\frac{1}{2}g_X(z)dz\otimes d\bar z.$ 
The rough Laplacian with respect to \(g_X\) is locally given by \(\Delta_{g_X} = \frac{4}{g_X(z)} \partial_z \partial_{\bar{z}}\). Let \(\Lambda_{\omega_X}\) denote the contraction with respect to \(\omega_X\).

The Gauss curvature of $g_X$ is $$k_{g_X}=-\frac{2}{g_X}\partial_z\partial_{\bar{z}}\log g_X.$$ Then $g_X$ satisfies $\partial_z\partial_{\bar{z}}\log g_X(z)=\frac{1}{2}g_X(z).$

Since \(X\) is non-compact, we can choose a Hermitian metric \(\tilde{h}_i\) on \(L_i\) such that 
\[
\sqrt{-1} F(\nabla_{\tilde{h}_i}) = \frac{n+1-2i}{2} \omega_X,
\]
and locally,
\[
\partial_z \partial_{\bar{z}} \log \tilde{h}_i(z) = -\frac{n+1-2i}{4} g_X(z).
\]
Let \(\tilde{h} = \text{diag}(\tilde{h}_1, \ldots, \tilde{h}_n)\) be a background Hermitian metric on \(E\) satisfying \(\det(h) = \det(\tilde{h})\). Then there exists an \(n\)-tuple of smooth real-valued functions on \(X\), \(w = (w_1, \ldots, w_n)\), such that \(h_i = e^{w_i} \tilde{h}_i\). Note that \(\det(h) = \det(\tilde{h})\) if and only if \(\sum_{k=1}^n w_k = 0\).

Denote by \(\|\gamma_i\|_{\tilde{h}, g_X}^2\) the squared norm of \(\gamma_i\) in terms of \(\tilde{h}\) and \(g_X\). Locally,
\begin{equation}\label{k_i}
    \|\gamma_i\|_{\tilde{h}, g_X}^2 = \frac{2|\gamma_i(z)|^2 \tilde{h}_i(z)^{-1} \tilde{h}_{i+1}(z)}{g_X(z)}.
\end{equation}
Let \(u_i = -\sum_{k=1}^i w_k\) for \(i = 1, \ldots, n-1\). Set \(u_0 = 0\) and, by assumption, \(u_n = 0\).

We now present the following proposition.

\begin{prop}\label{prop 1}
1. We have
  \begin{equation}\label{supersolution global fil eqn}
    -\sqrt{-1} \Lambda_{\omega_X} F_{\det(E/G_k)} + ||\gamma_k||_{h, g_X}^2\leq 0, \quad k = 1, \ldots, n-1.
  \end{equation} 
  Equivalently, in terms of \((u_1, \ldots, u_n)\) and the background metric \(\tilde{h}\), we have
  \begin{equation}\label{eqn3}
    \Delta_{g_X} u_i + i(n-i) - 2||\gamma_i||_{\tilde{h}, g_X}^2 e^{2u_i - u_{i-1} - u_{i+1}} \geq 0, \quad i = 1, \ldots, n-1, \quad u_0 = u_n = 0.
  \end{equation}
2. Suppose \((E,\theta)\) is a CVHS and \(h = \oplus h|_{L_i}\) is a diagonal harmonic metric. We have 
  \begin{equation}\label{global fil eqn}
    -\sqrt{-1} \Lambda_{\omega_X} F_{\det(E/G_k)}  + ||\gamma_k||_{h, g_X}^2 = 0, \quad k = 1, \ldots, n-1.
  \end{equation}
  Equivalently, in terms of \((u_1, \ldots, u_n)\) and the background metric \(\tilde{h}\), we have
  \begin{equation}\label{eqn2}
    \Delta_{g_X} u_i + i(n-i) - 2||\gamma_i||_{\tilde{h}, g_X}^2 e^{2u_i - u_{i-1} - u_{i+1}} = 0, \quad i = 1, \ldots, n-1, \quad u_0 = u_n = 0.
  \end{equation}
\end{prop}

\begin{proof}
For a holomorphic $\theta$-invariant subbundle $F$ of $E$, let $F^{\perp}$ denote the subbundle of $E$ orthogonal to $F$ with respect to the harmonic metric $h$. $F^{\perp}$ can be equipped with the quotient holomorphic structure from $E/F$. Considering the smooth orthogonal decomposition
$$
E = F^{\perp} \oplus F,
$$
we have the expressions for the holomorphic structure $\bar{\partial}_E$ and the Higgs field $\theta$ as
\begin{equation*}
    \bar{\partial}_E = \begin{pmatrix}
    \bar{\partial}_{F^{\perp}} & 0 \\
    \beta & \bar{\partial}_{F}
    \end{pmatrix}, \quad \theta = \begin{pmatrix}
    \theta_1 & 0 \\
    B & \theta_2
    \end{pmatrix}, \quad h = \begin{pmatrix}
    h_1 & 0 \\
    0 & h_2
    \end{pmatrix},
\end{equation*}
where $\beta \in \Omega^{0,1}(X, \text{Hom}(F^{\perp}, F))$ and $B \in \Omega^{1,0}(X, \text{Hom}(F^{\perp}, F))$. 

Denote $\bar\partial_E^0=\begin{pmatrix}
    \bar{\partial}_{F^{\perp}} & 0 \\
 0 & \bar{\partial}_{F}
    \end{pmatrix}.$

Under a holomorphic frame $(s^{\perp}, s)$ with respect to $(\bar{\partial}_{F^{\perp}}, \bar{\partial}_{F})$ (note that $s^{\perp}$ is not holomorphic with respect to $\bar{\partial}_E$), the connection 1-form of the Chern connection $\nabla_{\bar{\partial}_E, h}$ is given by $\omega + \beta$, where
$$
\omega = h^{-1} \partial h - h^{-1} \beta^* h.
$$
Here, we use $\beta$ to also denote $\begin{pmatrix}
0 & 0 \\
\beta & 0
\end{pmatrix}$ (similarly for $B$).

For a $(1,1)$-tensor $A$, such as $\theta$, $B$, or $\beta$, in a local frame, $A^{*_h} = h^{-1} A^* h$. Thus, $\beta^{*_h} = h_1^{-1} \beta^* h_2$. Calculations show that the Chern connection $\nabla_{\bar{\partial}_E, h}$ and the adjoint $\theta^{*_h}$ of the Higgs field $\theta$ are
\begin{equation*}
    \nabla_{\bar{\partial}_E, h} = \begin{pmatrix}
    \nabla_{\bar{\partial}_{F^{\perp}}, h_1} & -\beta^{*_h} \\
    \beta & \nabla_{\bar{\partial}_{F}, h_2}
    \end{pmatrix}, \quad
    \theta^{*_h} = \begin{pmatrix}
    \theta_1^{*_{h_1}} & B^{*_h} \\
    0 & \theta_2^{*_{h_2}}
    \end{pmatrix}.
\end{equation*}
The Chern curvature is given by
$$
F(\nabla_{\bar{\partial}_E, h}) = d(\omega + \beta) + (\omega + \beta) \wedge (\omega + \beta).
$$
Restricting to the $\text{Hom}(F^{\perp}, F^{\perp})$ component, we have
$$
F(\nabla_{\bar{\partial}_E, h})|_{\text{Hom}(F^{\perp}, F^{\perp})} = F(\nabla_{\bar{\partial}_{F^{\perp}}, h_1}) - \beta^{*_h} \wedge \beta.
$$
Thus, restricting the Hitchin equation to $\text{Hom}(F^{\perp}, F^{\perp})$, we obtain
\begin{equation*}
    F(\nabla_{\bar{\partial}_{F^{\perp}}, h_1}) - \beta^{*_h} \wedge \beta + B^{*_h} \wedge B + [\theta_1, \theta_1^{*_{h_1}}] = 0.
\end{equation*}
Taking the trace and noting that $\text{tr}([\theta_1, \theta_1^{*_{h_1}}]) = 0$, we obtain
\begin{equation*}
    \text{tr}(F(\nabla_{\bar{\partial}_{F^{\perp}}, h_1})) - \text{tr}(\beta^{*_h} \wedge \beta) + \text{tr}(B^{*_h} \wedge B) = 0.
\end{equation*}
Contracting with the Kähler form $\omega$, we have
$$-\sqrt{-1}\Lambda_{\omega_X} \text{tr}(F(\nabla_{\bar\partial_{F^{\perp}}, h_1}))-\sqrt{-1}\Lambda_{\omega_X} \text{tr}(B^{*_h}\wedge B)=-\sqrt{-1}\Lambda_{\omega_X}\text{tr}(\beta^{*_h}\wedge \beta)\leq 0.$$
Notice that $\text{tr}(F(\nabla_{\bar\partial_{F^{\perp}}, h_1}))$ is the curvature form $F_{\det F^{\perp}}$ of the determinant line bundle of $F^{\perp}$ with respect to the induced metric.
Therefore, we obtain the following inequality,
\begin{equation}\label{standard}
-\sqrt{-1}\Lambda_{\omega_X} F_{\det F^{\perp}}-\sqrt{-1}\Lambda_{\omega_X} \text{tr}(B^{*_h}\wedge B)\leq 0.
\end{equation}
The equality holds if and only if $\beta=0.$

Now for each $k=1,\cdots, n-1$, let $F$ be the subbundle $G_k=\bigoplus_{i\geq k+1}L_i$. Then we have $F^{\perp}=\bigoplus_{i\leq k}L_i$. The associated factor $B$ is given by
$$B=\begin{pmatrix}a_{k+1,1}&a_{k+1,2}&\cdots&a_{k+1,k-1}&\gamma_k\\a_{k+2,1}&a_{k+2,2}&\cdots&a_{k+2,k-1}&a_{k+2,k}\\\vdots&\vdots&\cdots&\vdots&\vdots\\a_{n,1}&a_{n,2}&\cdots&a_{n,k-1}&a_{n,k}\end{pmatrix},$$
mapping $F^{\perp}$ to $F\otimes K_{X}.$
Then
\begin{eqnarray}\label{B}
 \sqrt{-1}\Lambda_{\omega_X} \text{tr}(B^{*_h}\wedge B)\leq -2|\gamma_k|^2h_k^{-1}h_{k+1}/g_X.   
\end{eqnarray}
The equality holds if and only if all $a_{ij}=0.$ 

Since $\det(F^{\perp})=\det(E/G_k)$, together with (\ref{standard}), (\ref{B}), we obtain (\ref{supersolution global fil eqn}).

Locally, the induced metric on $\det F^{\perp}$ is given by $\prod\limits_{i=1}^k h_i(z).$ 
Then locally
\begin{equation}\label{F}
    -\sqrt{-1}\Lambda_{\omega_X} F_{\det F^{\perp}}=\frac{2}{g_X(z)}\partial_z\partial_{\bar{z}}\log (\prod\limits_{i=1}^kh_i(z)).
\end{equation}
Putting (\ref{supersolution global fil eqn}) and (\ref{F}) together, we obtain \begin{equation}\label{supersolution}   \partial_z \partial_{\bar{z}} \log \big( \prod_{i=1}^k h_i(z) \big) + |\gamma_k(z)|^2 h_k(z)^{-1} h_{k+1}(z) \leq 0, \quad k = 1, \ldots, n-1.\end{equation}
Thus, for $k=1,\cdots, n-1,$
\begin{equation}\label{Hloc}
    \partial_z \partial_{\bar{z}} \big( \sum_{k=1}^i w_k \big) + \partial_z \partial_{\bar{z}} \log \big( \prod_{k=1}^i \tilde{h}_k(z) \big) + |\gamma_i(z)|^2 e^{-w_i + w_{i+1}} \tilde{h}_i(z)^{-1} \tilde{h}_{i+1}(z) = 0.
\end{equation}
Applying the formula $\partial_z \partial_{\bar{z}} \log \tilde{h}_i(z) = -\frac{n+1-2i}{4} g_X(z)$, $u_i=-\sum_{k=1}^iw_k$, and note that \[\sum_{k=1}^i (n+1-2k) = i(n-i),\] we obtain (\ref{eqn3}).

When $(E,\theta)$ is a CVHS and $h$ is a diagonal harmonic metric, then all $a_{ij}$ vanishes and $\beta=0$. Thus we obtain the equality in Equation (\ref{supersolution global fil eqn}) and thus the equality in Equation (\ref{eqn3}).
\end{proof}

\subsection{$n$-Fuchsian case} \label{subsection:Fuchsian}
Suppose $L_i=K_{X}^{\frac{n+1-2i}{2}}$, then $\gamma_i$'s are constants. Recall the induced Hermitian metric on $K_{X}^{-1}$ is $\hat g_X=\frac{1}{2}g_X(z)dz\otimes d\bar{z}$.  Consider $\tilde{h}_i=\hat g_X^{-\frac{n+1-2i}{2}}.$ Locally, let $s_i=dz^{\frac{n+1-2i}{2}}$. Then $\tilde{h}_i(z)=(\frac{1}{2}g_X(z))^{-\frac{n+1-2i}{2}}.$ And $$\partial_z\partial_{\bar{z}}\log \tilde{h}_i=-\frac{n+1-2i}{2}\partial_z\partial_{\bar{z}}\log g_X(z)=-\frac{n+1-2i}{4}g_X(z).$$
In this case, if $\tilde{h}_i$ is a harmonic metric, i.e., the conformal factor $w_i$'s are constants, then the graded Higgs bundle can be given by $\gamma_i=\sqrt{\frac{i(n-i)}{2}}$. So $||\gamma_i||_{\tilde h, g_X}^2=\frac{i(n-i)}{2}.$

In this case, for a diagonal harmonic metric $h$, using $\tilde h$ in the above choice,  $(u_1,\cdots, u_n)$ satisfies
\begin{equation}\label{eqn4}
    \Delta_{g_X} u_i + i(n-i) - i(n-i) e^{2u_i - u_{i-1} - u_{i+1}} = 0, \quad i = 1, \ldots, n-1, \quad u_0 = u_n = 0.
\end{equation}
It is obvious that $(0,\cdots,0)$ is a solution to the system and thus $\tilde h=\oplus_{i=1}^n \hat g_X^{-\frac{n+1-2i}{2}}$ is a harmonic metric. Thus, \[||\theta||_{\tilde h, g_X}^2=\sum_{i=1}^{n-1}||\gamma_i||_{\tilde h, g_X}^2=\sum_{i=1}^{n-1}\frac{i(n-i)}{2}=\frac{n(n^2-1)}{12}.\]

\section{Generalization of Kalka-Yang's theorem to system}\label{Kalka-Yang}
 Consider the system of equations
\begin{equation}\label{maineq}
    \Delta_{g_X} u_i + i(n-i) - k_i e^{2u_i - u_{i-1} - u_{i+1}} = 0, \quad 1 \leq i \leq n-1, \quad u_0 = u_n = 0.
\end{equation}

\begin{definition}
1. An $n$-tuple of smooth functions \(u = (u_1, \ldots, u_{n-1})\) is called a subsolution of \eqref{maineq} if it satisfies \begin{equation*}
    \Delta_{g_X} u_i + i(n-i) - k_i e^{2u_i - u_{i-1} - u_{i+1}}\geq 0, \quad 1 \leq i \leq n-1, \quad u_0 = u_n = 0.
\end{equation*}
2. An $n$-tuple of smooth functions \(u = (u_1, \ldots, u_{n-1})\) is called a supersolution of \eqref{maineq} if it satisfies \begin{equation*}
    \Delta_{g_X} u_i + i(n-i) - k_i e^{2u_i - u_{i-1} - u_{i+1}}\leq 0, \quad 1 \leq i \leq n-1, \quad u_0 = u_n = 0.
\end{equation*}
\end{definition}
\begin{definition}
A solution \(u = (u_1, \ldots, u_{n-1})\) of \eqref{maineq} is called maximal if, for any subsolution \(v = (v_1, \ldots, v_{n-1})\) of \eqref{maineq}, we have \(u_i \geq v_i\) for all \(i\). Clearly, the maximal solution is unique.
\end{definition}

\begin{definition}\label{essentially positive}
A tuple of smooth nonnegative functions $(k_1,\cdots,k_{n-1})$ on $X$ is called simultaneously essentially positive if there is an exhaustion of $X$ by smooth bounded domains $\{\Omega_j\}_{j=1}^{\infty}$ such that $k_i>0$ on $\partial \Omega_j$ for every $i=1,\cdots,n-1$ and $j=1,2,\cdots.$
\end{definition}

We always assume \((X, g_X)\) is a non-compact complete hyperbolic surface. In this section, we show the following result:

\begin{prop}\label{prop:main1}
Let $(k_1,\cdots, k_{n-1})$ be a tuple of smooth  simultaneously essentially positive functions on \(X\). 
Suppose \eqref{maineq} has a subsolution $\underline{u} = (\underline{u}_1, \ldots, \underline{u}_{n-1})$, then it has a maximal solution $u=(u_1,\cdots,u_{n-1})$.

Moreover, either \(\underline{u}_i < u_i\) for every \(i = 1, \ldots, n-1\), or \(\underline{u}_i \equiv u_i\) for every \(i = 1, \ldots, n-1\).
\end{prop}

When \(n = 2\), Equation \eqref{maineq} simplifies to
\begin{equation}\label{eqn:single}
    \Delta_{g_X} u + 1 - k e^{2u} = 0.
\end{equation}
This is the prescribed curvature equation, which seeks a function \(u\) such that the metric \(e^{2u} \cdot g_X\) has curvature \(-k\). In this context, Kalka and Yang proved the following theorem, which Proposition \ref{prop:main1} generalizes:

\begin{theorem}[Kalka-Yang \cite{KalkaYang}]
Let \(k\) be an essentially positive smooth function on \(X\). If \eqref{eqn:single} has a subsolution, then it has a maximal solution.
\end{theorem}

Note that when $k=1$, $u=0$ is obviously a solution to Equation (\ref{eqn:single}). In fact, it is the maximal solution, which corresponds to the fact that $g_X$ is the maximal metric among all hyperbolic metric in its conformal class. As a generalization of this fact, we also prove the following proposition.
\begin{prop}\label{prop:main2}
When $k_i=i(n-i)$, $(0,\cdots,0)$ is the maximal solution to (\ref{maineq}).
\end{prop}
We delay the proof of Proposition \ref{prop:main1} and Proposition \ref{prop:main2} to the end of this section. We use Kalka-Yang's strategy \cite{KalkaYang} to prove Proposition \ref{prop:main1}. 

We will frequently use the following comparison principle:

\begin{lemma}\label{lem:comparisonprincipal} (Comparison principle)
Let \(\Omega\) be a bounded domain in \(X\). Let \(\underline{u} = (\underline{u}_1, \ldots, \underline{u}_{n-1})\) be a subsolution and \(\overline{u} = (\overline{u}_1, \ldots, \overline{u}_{n-1})\) be a supersolution to Equation \eqref{maineq} on \(\Omega\). Suppose \(\overline{u}_i \geq \underline{u}_i\) on \(\partial \Omega\) for every \(i = 1, \ldots, n-1\). Then either \(\underline{u}_i < \overline{u}_i\) for every \(i = 1, \ldots, n-1\), or \(\underline{u}_i \equiv \overline{u}_i\) for every \(i = 1, \ldots, n-1\).
\end{lemma}

\begin{proof}
Let \(u_i = \underline{u}_i - \overline{u}_i\) for \(1 \leq i \leq n-1\). Then \(u_i\) satisfies
\begin{equation}\label{linearization}
    \Delta_{g_X} u_i - c_i (2u_i - u_{i-1} - u_{i+1}) = 0,
\end{equation}
where
\[ c_i = k_i \int_{0}^1 e^{t(2\underline{u}_i - \underline{u}_{i-1} - \underline{u}_{i+1}) + (1-t)(2\overline{u}_i - \overline{u}_{i-1} - \overline{u}_{i+1})} \, dt. \]
We can verify that the above system satisfies the assumptions (a), (b), (c), and (d) in Lemma \ref{maximum principle} (assumption (c) is satisfied using the procedure in Remark \ref{assumption (c)}). This completes the proof.
\end{proof}

We will also use the method of super-subsolution for equation systems, which is prepared in the following lemma by Li-Mochizuki.
\begin{lemma}(\cite{LiMochizuki})
Let $(M,g)$ be a Riemannian manifold. Let $\bar{\Omega}$ be a bounded domain of $M$ with smooth boundary. Denote $\Omega$ as its interior. Let $F(x,y_1,\cdots,y_n)$ be a smooth function on $\Omega\times \mathbb{R}^n$ which satisfies $\frac{\partial F_k}{\partial y_j}\leq 0$ for $j\neq k$. Let $f=(f_1,\cdots,f_n)$ be continuous functions on $\partial \Omega$. Consider the system, for $k=1,\cdots, n$,
\begin{eqnarray*}
    \Delta_g u_k&=&F_k(x,u_1,\cdots,u_n) \quad \text{in }\Omega,\\
    u_k&=&f_k\qquad \qquad\qquad \quad \text{on }\partial \Omega.
\end{eqnarray*}
Suppose $\underline{u}$ and $\overline{u}$ are a subsolution and a
supersolution on $\bar{\Omega}$ respectively, satisfying $\underline{u}\leq \overline{u}$ in $\Omega$ and $\underline{u}\leq f\leq \overline{u}$ on $\partial \Omega$. Then there exists a smooth solution $u$ to the Dirichlet problem, satisfying $\underline{u}\leq u\leq \overline{u}.$
\end{lemma}
\begin{remark}
    In \cite[Proposition 5.2]{LiMochizuki}, Li-Mochizuki demonstrated the method of super-subsolution for the non-compact manifolds. In the Step 1 of their proof, they showed the method of super-subsolution for the Dirichlet problem.
\end{remark}

Let $\bar{\Omega}$ be a bounded domain of $X$ with smooth boundary. First, we show the existence of solution to Equation (\ref{maineq}) when $k$ is a smooth essentially positive function on $\bar\Omega$ and positive on $\partial \Omega.$
\begin{lemma}\label{step1} Let $\bar{\Omega}$ be a bounded domain of $X$ with smooth boundary.
    Let $k_i$ be smooth functions on $\bar{\Omega}$, $1\leq i\leq n-1$. Suppose $k_i\geq 0$ on $\bar{\Omega}$ and $k_i>0$ on $\partial \Omega$, $1\leq i\leq n-1.$ Then Equation (\ref{maineq}) has a smooth solution $u=(u_1,\cdots,u_{n-1})$ on $\Omega$ such that $\lim\limits_{x\in\partial\Omega}u_i(x)=+\infty$, $1\leq i\leq n-1$.
\end{lemma}
\begin{proof}
    We use the method of super-subsolution to obtain a solution. From the assumption, there is a constant $C_0>1$ such that $|k_i|\leq C_0$ for every $1\leq i\leq n-1$. Then $\underline{u}_i=-\frac{i(n-i)}{2}\ln C_0$, $1\leq i\leq n-1$, gives a subsolution. Now we construct a supersolution. Let $\Omega'$ be a bounded domain in $X$ such that $\bar{\Omega}\subset\Omega'.$ Consider the eigenvalue problem
    \begin{eqnarray*}
        -\Delta_{g_X} u&=&\mu u\quad \text{ on $\Omega'$},\\
        u&=&0\quad \text{ on $\partial\Omega'$}.
    \end{eqnarray*}
    Let $\lambda$ be the first eigenvalue and $\varphi$ be the eigenfunction of $\lambda$. It is well known that $\lambda>0$ and $\varphi$ is smooth and positive on $\Omega'$. Denote $C_1=\min\limits_{\bar{\Omega}}\varphi>0$. For each $N\in\mathbb{N}$, denote $u_i^N=NC_1^{-1}\varphi$. Then suppose $N\geq n(n-1)\lambda^{-1},$ for each $1\leq i\leq n-1,$
    $$\Delta_{g_X}u^N_i+i(n-i)-k_ie^{2u^N_i-u^N_{i-1}-u^N_{i+1}}\leq-\lambda u_i^N+i(n-i)\leq -\lambda N+i(n-i)\leq 0.$$
    Then $u_i^N\geq N$, $1\leq i\leq n-1$, gives a supersolution. It is clearly $u_i^N>\underline{u}_i$. 
    By the method of super-subsolution, the following boundary value problem, for each $1\leq i\leq n-1,$
    \begin{eqnarray}
         \Delta_{g_X}u_i+i(n-i)-k_ie^{2u_i-u_{i-1}-u_{i+1}}&=&0\quad\text{ on $\Omega$},\label{D_r}\\
        u_i&=&N\quad\text{ on $\partial\Omega$}.\nonumber
    \end{eqnarray}
    has a solution $v^N=\{v_1^N,\cdots,v_{n-1}^N\}$ such that for each $1\leq i\leq n-1,$
    $$-\frac{i(n-i)}{2}\ln C_0\leq v_i^N\leq u^N_i.$$
    From the comparison principle Lemma \ref{lem:comparisonprincipal}, we have $v^N$ is unique and increases with respect to $N$. 
    
    Now we show that for each $x_0\in \Omega$ and $i_0\in\{1,\cdots,n-1\}$, the sequence $\{v_{i_0}^N(x_0)\}_{N}$ is bounded. Denote $\Omega_{\epsilon}=\{x\in \Omega:~\text{dist}(x,\partial \Omega)>\epsilon\}.$ Since $k_i>0$ on $\partial\Omega$, $1\leq i\leq n-1$, and $\partial\Omega$ is compact, we may choose constants $\epsilon>0$ and $\delta>0$, such that $k_i>i(n-i)\delta$ on $\Omega\setminus\Omega_{\epsilon}$, $1\leq i\leq n-1$. We assume $\epsilon$ is smaller than the injective radius.

Note that $\Omega\setminus \Omega_{\epsilon}$ is again hyperbolic. There is a unique complete k\"ahler metric on $\Omega\setminus \Omega_{\epsilon}$ of constant curvature $-\delta$.
    
    For $x\in \Omega\setminus\Omega_{\frac{\epsilon}{2}}$, let $d=2\text{dist}(x,\partial \Omega).$ Then the hyperbolic disk $D_{d}(x)$ lies in $\Omega\setminus\Omega_{\epsilon}.$ 
    Regard $D_{d}(x)$ as a subset of the Poincar\'{e} disk model, $D_r=\{z\in \mathbb{C}:~|z|<r\}\subset \mathbb{D}^2$, where $r=\frac{e^d-1}{e^d+1}$ is the Euclidean length corresponding to the hyperbolic length $d$. Consider the complete metric $g_{\delta}$ of curvature $-\delta$ on $D_r$.   Explicitly, $g_{\delta}=\frac{4r^2|dz|^2}{\delta(r^2-|z|^2)^2}$. Let $g_\delta=e^{\tilde{w}}g_{\mathbb{D}}$, then $\tilde{w}=\ln \frac{r^2(1-|z|^2)^2}{\delta(r^2-|z|^2)^2}$. 
    Then $\tilde{w}$ satisfies $\Delta_{g_{\mathbb{D}}}\tilde{w}+2-2\delta e^{\tilde{w}}=0.$ Let $\tilde{w}_i=-\frac{n+1-2i}{2}\tilde{w}$ and $\tilde{u}_i=-\sum\limits_{k=1}^i\tilde{w}_k=\frac{i(n-i)}{2}\tilde{w}$, $1\leq i\leq n-1$. Then $\tilde{u}_i$ satisfies  \begin{eqnarray}\label{eqn r delta}
    \Delta_{g_{\mathbb{D}}}\tilde{u}_i+i(n-i)-i(n-i)\delta e^{2\tilde{u}_i-\tilde{u}_{i-1}-\tilde{u}_{i+1}}=0\quad\text{ on $D_r$}.  
    \end{eqnarray}
    Pulling back Equation (\ref{D_r}) on $D_d(x)$ to $D_r$, since $k_i>i(n-i)\delta$, we have   $$\Delta_{g_{\mathbb{D}}}v^N_i+i(n-i)-i(n-i)\delta e^{2v^N_i-v^N_{i-1}-v^N_{i+1}}\geq 0\quad\text{ on $D_r$}.$$
    Since $\tilde{u}_i$'s are $+\infty$ on $\partial D_r$, from the comparison principle Lemma \ref{lem:comparisonprincipal}, we obtain $v_i^N\leq \tilde{u}_i$. 
    In particular $v^N_i(x)\leq\tilde{u}_i(0)=\frac{i(n-i)}{2}\tilde{w}(0)=\frac{i(n-i)}{2}\ln\frac{1}{\delta r^2},$ $1\leq i\leq n-1$, are bounded with respect to $N$.

    For $x\in \Omega_{\frac{\epsilon}{4}}$, consider $v_{i}^{N_0}$ for a fixed $N_0\geq n(n-1)\lambda^{-1}$. As Equation (\ref{linearization}) in the proof of Lemma \ref{lem:comparisonprincipal}, $v_i^N-v_i^{N_0}$ satisfies the assumptions (a)(b) in the maximum principle Lemma \ref{maximum principle}. So the non-negative maximum of $v_i^N-v_i^{N_0}$ can obtain on the boundary of $\Omega_{\frac{\epsilon}{4}}$, which is bounded with respect to $N$ from the discussion above. So $v_i^N(x)$ is bounded with respect to $N$.

    Therefore, for each $1\leq i\leq n-1$, we have $v_i^N(x)$ is uniformly bounded on $\Omega_{\epsilon}$, and is increasingly convergent to a function $u_i(x)$ on $\Omega$. 
    From the Schauder theory and the bootstrap argument, $u_i$ is a smooth solution to Equation (\ref{maineq}) on $\Omega$.
    And it is clearly that $u_i$  approaches to $+\infty$ on $\partial \Omega.$ We finish the proof.
\end{proof}

Furthermore, we can always find the maximal solution under the assumptions in Lemma \ref{step1} from the following lemma.
\begin{lemma}\label{step2}
    Under the assumptions in Lemma \ref{step1}, Equation (\ref{maineq}) has a unique smooth maximal solution $v=(v_1,\cdots,v_{n-1})$ on $\Omega$ such that $\lim\limits_{x\in\partial\Omega}v_i(x)=+\infty$, $1\leq i\leq n-1$.
\end{lemma}
\begin{proof}
    We use the notation in the proof of Lemma \ref{step1}. For $\epsilon\geq 0$ small enough, let $u^{\epsilon}_i$, $1\leq i\leq n-1$, be a solution of Equation (\ref{maineq}) on $\Omega_{\epsilon}$ obtained from Lemma \ref{step1}. From the construction, for $\epsilon\geq 0$, we have $u_i^{\epsilon}\geq -\frac{i(n-i)}{2}\ln C_0$. From the comparison principle Lemma \ref{lem:comparisonprincipal}, for fixed $x\in\Omega$ and $i\in \{1,\cdots, n-1\}$, $u^{\epsilon}_i(x)$ decreases as $\epsilon$ decreases to $0$. 
    Set $v_i(x)=\lim\limits_{\epsilon \rightarrow 0}u_i^{\epsilon}(x).$ 
    (Notice that $v_i(x)$ may be different from $u^0_i(x)$.) From the Schauder theory and the bootstrap argument, $v_i$ is a smooth solution of Equation (\ref{maineq}) on $\Omega.$ For the boundary value, we have 
    $$\lim\limits_{x\rightarrow \partial\Omega}v_i(x)\geq \lim\limits_{x\rightarrow \partial\Omega}u^0_i(x)=+\infty.$$

    To see $v_i$ is the maximal solution, let $v'_i$ be another subsolution. From the comparison principle Lemma \ref{lem:comparisonprincipal}, we have $v'_i(x)\leq u^{\epsilon}_i(x)$ for all $\epsilon>0$ and $x\in \Omega_{\epsilon}.$ 
    Then for fixed $x\in \Omega$, letting $\epsilon\rightarrow 0$, 
    we obtain $$v'_i(x)\leq \lim\limits_{\epsilon\rightarrow 0} u^{\epsilon}_i(x)=v_i(x).$$
    We finish the proof.
\end{proof}
Now we are in the position to prove Proposition \ref{prop:main1}.
\begin{proof}[Proof of Proposition \ref{prop:main1}]
From the assumption, we may choose a bounded smooth exhaustion $\{\Omega_i\}_{i=1}^\infty$ of $X$ such that $k_i>0$, $i=1,\cdots,n-1,$ on $\partial \Omega_j,$ $\forall j.$ Let $\underline{u}=(\underline{u}_1,\cdots,\underline{u}_{n-1})$ be the subsolution. Then restricting to $\Omega_i$ $\underline{u}$ is also a subsolution. From Lemma \ref{step2}, there is a maximal solution $u^{(i)}$ on $\Omega_i$ with $\lim\limits_{x\in\partial\Omega}u^{(i)}(x)=+\infty.$ Then from the comparison principle Lemma \ref{lem:comparisonprincipal}, for each $j=1,\cdots,n-1$ and $x\in X$, $u^{(i)}_j(x)$ is decreasing with respect to $i$ ($i$ large enough), and $u^{(i)}_j(x)>\underline{u}_j(x)$. Set $$u_j(x)=\lim\limits_{i\rightarrow \infty}u_j^{(i)}(x),~j=1,\cdots,n-1, ~\forall x\in X.$$
From the Schauder theory and the bootstrap argument, $u=(u_1,\cdots,u_{n-1})$ is a smooth solution to Equation (\ref{maineq}) on $X$.

Now we show that $u$ is the maximal solution. Let $v$ be another subsolution to Equation (\ref{maineq}). For $x\in X$, suppose $x\in \Omega_{i}$ for $i\geq i_0$. Since $u^{(i)}$ is the maximal solution on $\Omega_i$, we have $u^{(i)}(x)\geq v(x)$. Then 
$$u(x)=\lim\limits_{i\rightarrow \infty}u^{(i)}(x)\geq v(x).$$
We finish the proof of $u$ being maximal.

At last, we show the ``moreover" part. We apply Lemma \ref{lem:comparisonprincipal} to the subsolution $\underline u$ and the maximal solution $u$ to each bounded domain $\Omega$. Then we have on $\Omega$, either \(\underline{u}_i < u_i\) for every \(i = 1, \ldots, n-1\), or \(\underline{u}_i \equiv u_i\) for every \(i = 1, \ldots, n-1\). Since $\Omega$ is arbitrary, we obtain the claim on $X$.
\end{proof}
Proposition \ref{prop:main1} shows the existence of the maximal solution. Now we prove Proposition \ref{prop:main2}, i.e. for the base $n$-Fuchsian case, the solution given by the hyperbolic metric is just the maximal solution.
\begin{proof}[Proof of Proposition \ref{prop:main2}]
   By lifting, we only need to show the proposition in the case $X=\mathbb{D}$. Let $u$ be a solution to Equation (\ref{eqn4}). Denote $D_r=\{z\in \mathbb{C}:~|z|<r\}$. For $r<1$, restricting on $D_r$, $u$ also satisfies Equation (\ref{eqn4}). Fix $z_0\in \mathbb{D}$. Suppose $r$ approaches to $1$ enough such that $z_0\in D_r.$ Consider the complete hyperbolic metric $g_r$ on $D_r$, $g_r=\frac{4r^2|dz|^2}{(r^2-|z|^2)^2}$. With respect to the background metric $g_X$, $g_r=e^{\ln \frac{r^2(1-|z|^2)}{(r^2-|z|^2)^2}}g_X$. As in Equation (\ref{eqn r delta}), $\tilde{u}_i^r=\frac{i(n-i)}{2}\ln \frac{r^2(1-|z|^2)}{(r^2-|z|^2)^2}$ is a solution to Equation (\ref{eqn4}) on $D_r$. Then by the comparison principle Lemma \ref{lem:comparisonprincipal}, \[u_i\leq \tilde u_i^r.\] Letting $r\rightarrow 1$, we obtain $u_i\leq 0$.
\end{proof}

\section{Main theorems and the proof}\label{sec:ProofofMainTheorem}


\begin{definition}
On a generically regular nilpotent Higgs bundle $(E,\theta)$, we say that a harmonic metric $h$ is maximal if it weakly dominates all other harmonic metrics. 
\end{definition}

In this section, we show the following theorem. 
\begin{theorem}\label{main theoremLater}
Let $(E,\theta)$ be a generically regular nilpotent Higgs bundle over a hyperbolic Riemann surface $X$. Suppose there is a harmonic metric $h$ on $(E,\theta)$. Then there uniquely exists a maximal harmonic metric $h_{max}$ on the graded Higgs bundle $\text{Gr}_\mathcal{G}(E,\theta)$ satisfying (1) $\det(h_{max})=\det(h)$, and (2) $h_{max}$ strictly weakly dominates $h$, unless $(E,\theta)$ is a CVHS of type $(1,\cdots,1)$ and $h=h_{\max}$. 

Furthermore, $h_{max}$ is diagonal, that is, $h_{max}=\oplus_{i=1}^nh_{max}|_{Gr_iE}$.
\end{theorem}
\begin{proof}
Write $(E,\bar\partial_E,\theta)$ in terms of the smooth decomposition $E=\oplus_{i=1}^nL_i$ with respect to $h$ as in Section \ref{generically regular} so that $h=\oplus_i h|_{L_i}$ and $G_i=\oplus_{k\geq i}L_i$ (smooth direct sum). 

Note that $L_i\cong G_i/G_{i+1}$. The graded Higgs bundle of $(E,\bar\partial_E, \theta)$ with respect to $\mathcal G$ is isomorphic to \[E_0=\oplus_{i=1}^nL_i, \quad \bar\partial_{E_0}= \oplus_{i=1}^n\bar\partial_{L_i},\quad \theta_0=\sum_{i=1}^{n-1}\gamma_i.\]  The metric $h$ on $E$ induces a natural metric, still denoted by $h$, on $E_0$ under the canonical smooth isomorphism $E\cong E_0$. Let $\tilde h$ be a background diagonal Hermitian metric on $E$ satisfying $\det(\tilde h)=\det(h)$. 

Write $h|_{L_i}=e^{w_i}\cdot \tilde h|_{L_i}$ and $u_i=-\sum_{k=1}^iw_k$. By Proposition \ref{prop 1}, we obtain 
\begin{equation*}
    \Delta_{g_X} u_i + i(n-i) - 2||\gamma_i||_{\tilde{h}, g_X}^2 e^{2u_i - u_{i-1} - u_{i+1}} \geq 0, \quad i = 1, \ldots, n-1, \quad u_0 = u_n = 0.
\end{equation*}

Since for every $i=1,\cdots,n-1$, $\gamma_i$ is holomorphic and not identically $0$, we see that $\gamma_i$ has isolated zeros.
Then the tuple $(||\gamma_1||_{\tilde h, g_X}^2,\cdots,||\gamma_{n-1}||_{\tilde h, g_X}^2)$ is simultaneously essentially positive in the sense of Definition \ref{essentially positive}. Thus one may apply Proposition \ref{prop:main1} and obtain there is a maximal solution to \begin{equation*}
    \Delta_{g_X} u_i + i(n-i) - 2||\gamma_i||_{\tilde{h}, g_X}^2 e^{2u_i - u_{i-1} - u_{i+1}}=0, \quad i = 1, \ldots, n-1, \quad u_0 = u_n = 0.
  \end{equation*}
From Proposition \ref{prop 1}, the maximal solution to the above equation gives a harmonic metric $\hat h$, which is diagonal, on the graded Higgs bundle satisfying $\det(\hat h)=\det(\tilde h)=\det(h)$.  

By the definition of $u_i$, we obtain $\hat h$ strictly weakly dominates $h.$

On $Gr_{\mathcal G}(E,\theta)$, given any (not necessarily diagonal) harmonic metric $h$, we can apply the above argument. And thus obtain the above $\hat h$ weakly dominates $h.$ 
Therefore, $\hat h$ is maximal among all harmonic metrics with the same determinant on the graded Higgs bundle. 

Thus, $\hat h$ is the desired $h_{max}$ in the statement.
\end{proof}

\begin{theorem}\label{main theoremLater1}
The metric $h_X=\oplus_{i=1}^n \hat g_X^{-\frac{n+1-2i}{2}}$ is the maximal harmonic metric on the base $n$-Fuchsian Higgs bundle.
\end{theorem}
\begin{proof} The statement directly follows from Section \ref{subsection:Fuchsian} and Proposition \ref{prop:main2}.
\end{proof}

\section{Harmonic metrics on  CVHS}\label{sec:non-existence}
Given a generically regular nilpotent Higgs bundle \((E, \theta)\), Theorem \ref{main theoremLater} establishes that the absence of a harmonic metric on the graded bundle \(\text{Gr}_{\mathcal{G}}(E, \theta)\) implies the nonexistence of a harmonic metric on \((E, \theta)\). We discuss the existence of harmonic metrics on a CVHS of type \((1, \ldots, 1)\) (see Definition \ref{complex variation of Hodge structure}).

\textbf{In this whole section, we always assume $(E,\theta)$ is a CVHS of type $(1,\cdots, 1)$.}

Let \(\Sigma\) be a compact Riemann surface of genus at least 2, with \(\pi: \Sigma_1 \rightarrow \Sigma\) is a holomorphic covering.
\begin{lemma}\label{lem:ExistenceLift}
 Given $(E,\theta)$ on $\Sigma$ and its lift $\pi^*(E,\theta)$ to $\Sigma_1$, then $(E,\theta)$ admits a harmonic metric if and only if  $\pi^*(E,\theta)$ admits a harmonic metric. 

In this case, both $(E,\theta)$ and $\pi^*(E,\theta)$ admit diagonal harmonic metrics.
\end{lemma}
\begin{proof}
It is clear that one only needs to show the case that $\pi$ is a universal covering map.

The ``only if" direction is obvious. 

We only need to consider the ``if" direction. Suppose that a harmonic metric exists on \(\pi^*(E, \theta)\). Since \(\pi^*(E, \theta)\) is a CVHS of type \((1, \cdots, 1)\), it possesses a unique maximal solution \(h\) as guaranteed by Theorem \ref{main theoremLater}. Given that \(\pi^*(E, \theta)\) is \(\pi_1(\Sigma)\)-invariant, the uniqueness of the maximal harmonic metric implies that \(h\) must also be \(\pi_1(\Sigma)\)-invariant, so it descends to a harmonic metric on $(E,\theta)$. 

The last statement follows from the fact that the maximal harmonic metric is diagonal.
\end{proof}

Let $\mathbb D$ denote the Poincaré disk. A CVHS of type \((1, \ldots, 1)\) over $\mathbb D$ is given by:
\[
   (E, \theta) = \left(\mathcal{O} \oplus \cdots \oplus \mathcal{O},\quad 
   \begin{pmatrix}
       0 & & & \\
       \gamma_1 & 0 & & \\
       & \ddots & \ddots & \\
       & & \gamma_{n-1} & 0
   \end{pmatrix} dz\right),
\] where $\gamma_i (i=1,\cdots,n-1)$ are holomorphic functions on $\mathbb D$. 

For $n=2,$ a CVHS of type $(1,1)$ over $\mathbb D$ is given by:
\[
   (E, \theta) = \left(\mathcal{O} \oplus \mathcal{O},\quad 
   \begin{pmatrix}
       0 &0 \\
       \gamma & 0 \end{pmatrix} dz\right),
\] where $\gamma$ is a holomorphic function on $\mathbb D$.

Denote the weighted Bergman space by
\[
\mathcal{A} = \{f \text{ is a holomorphic function on } \mathbb{D} : \iint_{\mathbb{D}} |f(z)|^2 (1 - |z|^2) d\sigma < \infty \}.
\]
Here $d\sigma$ is the Lebesgue measure on $\mathbb D.$ Denote by $\tilde {\mathcal A}$ the space of holomorphic functions $\alpha$ on $\mathbb D$
 such that there exists a nonvanishing holomorphic function $g$ on $\mathbb D$ so that $g\alpha\in \mathcal A.$
\begin{lemma}(\cite[Theorem 1.3]{Kraus})\label{lem:Kraus}
For \(n=2\), \((E, \theta)\) over $\mathbb D$ admits a diagonal harmonic metric of unit determinant if and only if $\gamma\in \tilde{\mathcal A}$.
\end{lemma}

Here we can improve Kraus' result slightly.
\begin{prop}\label{prop:ExistenceGamma}
For \(n=2\), \((E, \theta)\) admits a harmonic metric of unit determinant if and only if $\gamma\in \tilde{\mathcal A}$.
\end{prop}
\begin{proof}
This follows from Theorem \ref{main theoremLater} and Lemma \ref{lem:Kraus}.
\end{proof}

As an illustration, we present a holomorphic function on \(\mathbb{D}\) that does not belong to \(\tilde{\mathcal{A}}\).  Let $\Sigma$ be a compact Riemann surface. Consider the Higgs bundle \((E, \theta)\) over \(\Sigma\) defined as:
\[
   E = L \oplus L^{-1}, \quad 
   \theta = \begin{pmatrix}
       0 & 0 \\
       \eta & 0
   \end{pmatrix}, \quad \eta \neq 0,
\] where $L$ is a holomorphic line bundle over $\Sigma$ and $\eta\in H^0(\Sigma, L^{-2}K_{\Sigma}).$
The existence of $\eta\neq 0$ implies $\deg(L)\leq g-1.$ After identifying $\pi^*L=\mathcal O,$ we obtain \(\pi^* \eta = \gamma dz,\) where $\gamma$ is a holomorphic function on $\mathbb D.$ We then have the following claim.
\begin{prop}\label{prop:NotInA}
$\deg(L)>0$ if and only if $\gamma\in \tilde{\mathcal A}.$

As a result, there exists $\gamma\notin \tilde{\mathcal A}.$
\end{prop}
\begin{proof}
The Higgs bundle \((E, \theta)\) is stable if and only if \(\deg L > 0\). By the Hitchin-Kobayashi correspondence for Higgs bundles on a compact Riemann surface, we obtain \(\deg L > 0\) if and only if $(E,\theta)$ admits a harmonic metric.

By Lemma \ref{lem:ExistenceLift}, we obtain \(\deg L > 0\) if and only if $\pi^*(E,\theta)$ admits a harmonic metric.

By Proposition \ref{prop:ExistenceGamma}, we obtain
\(\deg L > 0\) if and only if 
\(\gamma \notin \tilde{\mathcal{A}}\).
\end{proof}

\begin{lemma}(\cite[Theorem 7.3]{LiMochizukiHitchinSection})\label{lem:ExistenceCVHS}
 Suppose $\gamma_i\in \tilde{\mathcal A} (i=1,\cdots,n-1),$ $(E,\theta)$ admits a diagonal harmonic metric. 
\end{lemma}

\begin{lemma}(\cite[Proposition 7.10]{LiMochizukiHitchinSection})\label{lem:NonExistenceProduct}
Suppose \(\prod\limits_{i=1}^{n-1} \gamma_i^{i(n-i)} = \alpha^{\frac{n(n^2-1)}{6}}\) for a holomorphic function \(\alpha\) on \(\mathbb{D}\). If \((E, \theta)\) admits a diagonal harmonic metric, then \(\alpha \in \tilde{\mathcal{A}}\).
\end{lemma}

Again, we can slightly improve the above result. 
\begin{prop}\label{prop:NonExistenceProduct}
Suppose \(\prod\limits_{i=1}^{n-1} \gamma_i^{i(n-i)} = \alpha^{\frac{n(n^2-1)}{6}}\) for a holomorphic function \(\alpha\) on \(\mathbb{D}\). If \((E, \theta)\) admits a harmonic metric, then \(\alpha \in \tilde{\mathcal{A}}\).
\end{prop}
\begin{proof}
It follows from Theorem \ref{main theoremLater} and Lemma \ref{lem:NonExistenceProduct}.
\end{proof}

\begin{prop}\label{prop:NonExistenceSame}Suppose  $\gamma_1=\cdots=\gamma_{n-1}=\gamma.$ Then $(E,\theta)$ admits a harmonic metric if and only if $\gamma\in \tilde{\mathcal A}.$   
\end{prop}
\begin{proof}It follows from Lemma \ref{lem:ExistenceCVHS} and Proposition \ref{prop:NonExistenceProduct}.
\end{proof}

\begin{prop}\label{prop:NonExistence}
For each $n$, there exists a rank $n$ CVHS of type $(1,\cdots, 1)$ on $\mathbb D$ that does not accept a harmonic metric. 
\end{prop}
\begin{proof}
From Proposition \ref{prop:NotInA}, there exists $\alpha\notin\tilde{\mathcal A}.$ Choose a CVHS $(E,\theta)$ of type $(1,\cdots,1)$ parametrized by $(\alpha, \cdots, \alpha).$ It follows from Proposition \ref{prop:NotInA} and Proposition \ref{prop:NonExistenceSame}. 
\end{proof}

\section{Further discussion}\label{further discussion}

\subsection{$\mathbb{C}^*$-action}
Let \((E, \theta)\) be a generically regular nilpotent Higgs bundle. Consider the \(\mathbb{C}^*\)-action, i.e., for \(t \in \mathbb{C}^*\), \(t \cdot [(E, \theta)] = [(E, t\theta)]\). Here, \([\cdot]\) denotes the gauge equivalence class. Fixing a Hermitian metric \(h_0\), we can express \((E, \theta, h_0)\) in the form of (\ref{Metric}), (\ref{cpx str}), and (\ref{Higgs field}). Consider the gauge transformation \(g_t = \text{diag}(1, t, \ldots, t^{n-1})\). Then,
\[
g_t \cdot (\bar{\partial}_E, t\theta) = (g_t^{-1} \circ \bar{\partial}_E \circ g_t, g_t^{-1} \circ t\theta \circ g_t).
\]
Letting \(t\) approach \(+\infty\), we obtain the graded Higgs bundle \(\text{Gr}_{\mathcal{G}}(E, \theta)\), a CVHS of type $(1,\cdots,1)$. It does not depend on the initial choice of the metric $h_0.$

We can rephrase Theorem \ref{main theoremLater} as follows:
\begin{corollary}
Let \((E, \theta)\) be a generically regular nilpotent Higgs bundle with the canonical filtration $\mathcal G$. Suppose that there exists a harmonic metric \(h_0\) on \((E, \theta)\). Then there exists a maximal harmonic metric \(h\) on \(Gr_{\mathcal G}(E,\theta)=\lim\limits_{t \rightarrow \infty} g_t \cdot (E, t\theta)\) satisfying \(\det(h) = \det(h_0)\) and $h$
 weakly dominates $h_0$ with respect to $\mathcal G$.\end{corollary}

This result can be seen as a complement to Li-Mochizuki \cite[Theorem D]{LiMochizukiHitchinSection} for \(t\) approaching \(0\). 

\begin{ass}\label{assumption}
There exists an increasing full holomorphic filtration $\mathcal F=\{F_k\}$ of $E$ such that $\theta$ takes $F_k$ to $F_{k+1}\otimes K_X$ such that the induced maps on $F_k/F_{k-1}(k=1,\cdots,n-1)$ are non-zero. 
\end{ass}

Suppose $(E,\theta)$ satisfies Assumption \ref{assumption}. Fixing a Hermitian metric \(h_0\) on $E$, take $L_{i}$ to be the perpendicular line bundle of $F_{i-1}$ inside $F_{i}$ with respect to $h_0$, with the induced quotient holomorphic structure. We can express $E$ as a smooth direct sum: 
\[E=L_1\oplus L_2\oplus\cdots\oplus L_n.\]
With respect to this decomposition, the holomorphic structure $\bar\partial_E$ is strictly upper triangular and $\theta$ satisfies $\theta_{ij}=0$ for $i> j+1.$

Consider the gauge transformation \(g_t = \text{diag}(1, t, \ldots, t^{n-1})\) with respect to this smooth decomposition. Then,
\[
g_t \cdot (\bar{\partial}_E, t\theta) = (g_t^{-1} \circ \bar{\partial}_E \circ g_t, g_t^{-1} \circ t\theta \circ g_t).
\]
Letting \(t\) approach \(0\), we obtain the graded Higgs bundle \(\text{Gr}_{\mathcal{F}}(E, \theta)\), a CVHS of type $(1,\cdots,1)$. It does not depend on the initial choice of the metric $h_0.$
\begin{theorem}(\cite[Theorem D]{LiMochizukiHitchinSection})
Suppose $(E,\theta)$ satisfies Assumption \ref{assumption} with a filtration $\mathcal F$.  If $Gr_{\mathcal F}(E,\theta)=\lim\limits_{t\rightarrow 0}g_t\cdot (\bar\partial_E, t\theta)$ admits a diagonal harmonic metric $h_0$, then there exists a harmonic metric on $(E,\theta)$ which weakly dominates the metric $h_0$ with respect to $\mathcal F$. 
\end{theorem}

By Theorem \ref{main theoremLater}, we can slightly improve the result as follows. 
\begin{prop} Suppose $(E,\theta)$ satisfies Assumption \ref{assumption} with a filtration $\mathcal F$. 
 If $Gr_{\mathcal F}(E,\theta)=\lim\limits_{t\rightarrow 0}g_t\cdot (\bar\partial_E, t\theta)$ admits a harmonic metric $h_0$, then there exists a harmonic metric on $(E,\theta)$. 
\end{prop}

Next, we have the following statement. 
\begin{theorem}
For a generically regular nilpotent Higgs bundle $(E,\theta)$,  we obtain that $(E,\theta)$ satisfies Assumption \ref{assumption} with a filtration $\mathcal F$. Moreover, if $\text{Gr}_{\mathcal F}(E,\theta)=\lim\limits_{t\rightarrow 0}g_t\cdot (\bar\partial_E, t\theta)$ admits a diagonal harmonic metric $h_0$, then there exists a harmonic metric on $(E,\theta)$ which weakly dominates the metric $h_0$ with respect to $\mathcal F$. 
\end{theorem}
\begin{proof}
Since $\theta^{n-1}\neq 0$, there exists a point $p\in X$ and a vector $v\in E_p$ such that $\theta^{n-1}(v)\neq 0$. Hence $v, \theta(v),\cdots, \theta^{n-1}(v)$ span $E_p$. That is, $v$ is a cyclic vector of $E$ at $p$. By \cite[Proposition 4.26]{LiMochizukiHitchinSection}, we obtain that $(E,\theta)$ satisfies Assumption \ref{assumption} with a filtration $\mathcal F$.    
\end{proof}

\subsection{$SL(2,\mathbb C)$-case and minimal surfaces in $\mathbb H^3$}\label{n=2}
Let $(E,\theta)$, where $\theta\neq 0$, be an $SL(2,\mathbb{C})$-Higgs bundle over a non-compact hyperbolic Riemann surface $X$. In the context of non-Abelian Hodge theory, a harmonic metric corresponds to a $\rho$-equivariant harmonic map $f:\widetilde{X}\rightarrow SL(2,\mathbb{C})/SU(2)\cong \mathbb{H}^3$, where $\widetilde{X}$ is the universal cover of $X$, and $\rho:\pi_1(X)\rightarrow SL(2,\mathbb{C})$ is the holonomy representation of the flat connection $\nabla=\nabla_{h}+\theta+\theta^{*_h}$. The Hopf differential associated with this setup is given by $\text{tr}(\theta^2)$.

Now, assume $(E, \theta)$ is nilpotent. In this case, the Hopf differential vanishes, implying that the harmonic map $f$ is conformal and thus corresponds to a branched minimal surface in $\mathbb{H}^3$. The branched points of this surface coincides with the zero set of $\theta$, since the pullback metric is $\text{tr}(\theta\theta^{*_h})$. Conversely, any equivariant branched minimal disk in $\mathbb{H}^3$ arises from a nilpotent $SL(2,\mathbb{C})$-Higgs bundle $(E,\theta)$ equipped with a harmonic metric $h$. Then $(E,\theta,h)$ is of the following form:
\[E=L\oplus L^{-1}, \bar\partial_E=\begin{pmatrix}\bar\partial_L&0\\\beta&\bar\partial_{L^{-1}}\end{pmatrix}, \theta=\begin{pmatrix}
0 & 0\\
\gamma & 0
\end{pmatrix}, h=h|_L\oplus (h|_L)^{-1}.\]
The graded Higgs bundle with respect to the canonical filtration is
$$E^0=L\oplus L^{-1},\quad\bar\partial_{E^0}=\begin{pmatrix}\bar\partial_L&0\\0&\bar\partial_{L^{-1}}\end{pmatrix}\quad \theta^0=\left(
\begin{array}{cc}
0 & 0\\
\gamma & 0
\end{array}
\right).$$ According to Theorem \ref{main theoremLater}, $(E^0,\theta^0)$
has a diagonal harmonic metric \[h^0=h^0|_L\oplus h^0|_{L^{-1}}\] satisfying $h^0|_{L^{-1}}=(h^0|_L)^{-1}$ and $h^0|_L>h|_L$ unless $(E,\theta)=(E^0,\theta^0)$ and $h^0=h$. This corresponds to an equivariant harmonic map $f^0:\widetilde{X}\rightarrow SL(2,\mathbb{R})/SO(2)\cong \mathbb{D}.$ The Hopf differential of $f^0$ also vanishes. From \cite[Proposition on Page 10]{SchoenYauHarmonicMaps},  $f^0$ is either holomorphic or anti-holomorphic. By a reflection, we may assume $f^0$ is holomorphic. This leads us to the following result.

\begin{theorem}\label{branched minimal}
    Let \(f: \widetilde{X} \to \mathbb{H}^3\) be an equivariant branched minimal disk which does not lie in a totally geodesic copy of $\mathbb D$. Then it induces an equivariant holomorphic map \(f^0: \widetilde{X} \to \mathbb{D}\) satisfying \[0<\frac{g_f}{g_{f_0}}<1.\] In particular, the branched points of \(f\) coincide with the critical points of \(f^0\).
\end{theorem}
\begin{proof}
First, $g_f=\text{tr}(\theta\theta^{*_h})=|\gamma|^2(h|_L)^{-2}<|\gamma|^2(h^0|_L)^{-2}=\text{tr}(\theta^0(\theta^0)^{*_{h^0}})=g_{f^0}.$ Secondly, since the critical set of $f$ coincides with the one of $f^0$ with multiplicity counted, we have $\frac{g_f}{g_{f_0}}>0.$
\end{proof}

Note that an equivariant holomorphic map $f^0:\widetilde{X}\rightarrow\mathbb{D}$ automatically extends to an equivariant branched minimal disk in $\mathbb{H}^3$ by composing with a totally geodesic embedding $i:\mathbb{D}\rightarrow\mathbb{H}^3$. We can also state the result as follows. 
\begin{prop}\label{branched minimal1}
Let $D=\{z_j\}$ be a sequence in $\mathbb{D}$. Then the following statements are equivalent. \\
1. There is an equivariant branched conformal minimal map from $\tilde X\rightarrow \mathbb H^3$ with branch set $D$.\\
2. There is an equivariant holomorphic map from $\tilde X\rightarrow \mathbb H^2$ with critical set $D$.
\end{prop}
Kraus \cite{Kraus} characterized the set of critical points of a holomorphic map from $\mathbb{D}$ to itself. 
\begin{theorem}(\cite[Theorem 1.1]{Kraus})
Let $D=\{z_j\}$ be a sequence in $\mathbb{D}$. Then the following statements are equivalent.\\
(a) There is a holomorphic map from $\mathbb{D}$ to itself with critical set $D$.\\
(b) There is an indestructible Blaschke product with critical set $D$.\\
(c) There is a function in the weighted Bergman space $\mathcal{A}$ with zero set $D$.\\
(d) There is a holomorphic function in the Nevanlinna class $\mathcal{N}$ with critical set $D$.
\end{theorem}
A Blaschke product $B$ is of the form $B(z)=\prod\limits_{j=1}^{\infty}\frac{\bar{z}_j}{|z_j|}\frac{z_j-z}{1-\bar{z}_jz}.$ A Blaschke product $B$ is said to be indestructible if $T\circ B$ is a Blaschke product
for every unit disc automorphism $T$. 
A holomorphic function $f$ in $\mathbb{D}$ belongs to $\mathcal{N}$ if and only if the integrals $$\int_{0}^{2\pi}\log^+|f(re^{it})|dt$$
remains bounded as $r\rightarrow 1$.

\appendix
\section{Maximum Principle}
The key tool is the maximum principle for systems. The following maximum principle is well-known in the literature, for example \cite{LpezGmez1994TheMP}. We include the proof here for the convenience of the readers and later use.
\begin{lemma}\label{maximum principle}(Maximum Principle)
    Let $\bar{M}$ be a manifold with boundary. Denote $M$ as the interior of $\bar{M}$. Let $u_i$ be functions on $\bar{M}$, $1\leq i\leq n$, which are $C^2$ on $M$ and $C^0$ on $\bar{M}$. Let $c_{ij}$ be $C^0$ functions on $M$, $1\leq i,j\leq n$.
    Suppose $c_{ij}$ satisfy the following assumptions: \\
(a) cooperative: $c_{ij}\geq 0,~ i\neq j$,\\
(b) row diagonally dominant: $\sum\limits_{j=1}^{n}c_{ij}\leq 0,~ 1\leq i\leq n$.\\
Let $f_i$ be non-negative $C^0$ functions on $M$, $X_i$ be $C^0$ vector fields on $M$, and $g_i$ be Riemannian metrics on $M$, $1\leq i\leq n$ . Suppose $u_i$ satisfy
\begin{eqnarray*}
\Delta_{g_i} u_i+<X_i,\nabla_{g_i} u_i>+\sum_{j=1}^{n}c_{ij}u_j=f_i \text{ on } M, \quad 1\leq i\leq n.
\end{eqnarray*}
Then 
    $$\max_{x\in \bar{M}} \max\{u_i(x),0\}\leq \max_{x\in \partial M} \max\{u_i(x),0\},\quad 1\leq i\leq n.$$
\end{lemma}

\begin{lemma}\label{strong maximum principle}(Strong Maximum Principle)  Under the assumptions in Lemma \ref{maximum principle},
suppose $c_{ij}$ further satisfy the following assumption: \\
(c) fully coupled: the index set $\{1,\cdots,n\}$ cannot be split up in two disjoint nonempty sets $\alpha,\beta$ such that $c_{ij}\equiv 0$ for $i\in\alpha,j\in \beta.$\\
Then (a)(b)(c) implies that if for some $i_0$, $u_{i_0}$ attains its non-negative maximum at an interior point, then $u_{i}\equiv M$, $1\leq i\leq n$.
\\
Furthermore, either the following assumption ensures the constant $M$ must be $0$.\\
(d) $\max\limits_{x\in \partial{M}}u_i\leq 0$, $1\leq i\leq n.$
\\
(e) $(\sum\limits_{j=1}^nc_{1j},\cdots,\sum\limits_{j=1}^nc_{nj})\neq (0,\dots,0)$, i.e. there exists $i_{0}\in \{1,\cdots, n\}$, $x_0\in M$, such that $\sum\limits_{j=1}^nc_{i_0j}(x_0)\neq 0$.\\
Together with the assumptions (a)(b)(c), the following assumption rules out the possibility that $u_i$ can attain its non-negative maximum at interior points.\\
(f) $(f_1,\cdots,f_n)\neq (0,\dots,0)$, i.e. there exists $i_{0}\in \{1,\cdots, n\}$, $x_0\in M$, such that $f_{i_0}(x_0)\neq 0$.
\end{lemma}
\begin{proof}
We prove Lemma \ref{maximum principle} and Lemma \ref{strong maximum principle} together.
We first assume $\max\limits_{x\in \partial M}u_i\leq 0$, $1\leq i\leq n$, which is the assumption (d). 
Let $M_i$ be the maximum of $u_i$ on $\bar{M}$ and $M_0=\max\limits_{i}M_i$. Let $\alpha$ be the subset of $\{1,\cdots,n\}$ such that $i\in\alpha$ if and only if $M_i=M_0$ and $M_i$ is achieved by an interior point. Let $\beta$ be the complement of $\alpha$. If $\alpha$ is empty, then any $i$ such that $M_i=M_0$ satisfies $M_i$ is only achieved on $\partial M$ and thus $M_0\leq 0.$ Hence, we finish the proof. Now we suppose $\alpha$ is not empty.

Claim 1: under the assumptions (a)(b), either (1) $M_0<0$ or (2) $M_0=0$ and $u_i\equiv M_i=M_0=0$ for $i\in \alpha$, $c_{ij}\equiv 0$ for $i\in \alpha,j\in \beta$. 

Suppose for $i_0\in\alpha$, $M_0$ is achieved by an interior point $x_0$, i.e. $$M_0=u_{i_0}(x_0)=\max_{i}\max_{x\in \bar{M}}u_i.$$ 
Then $$\Delta_{g_{i_0}} u_{i_0}+<X_{i_0},\nabla_{g_{i_0}} u_{i_0}>+c_{i_0i_0}(u_{i_0}-M_0)+c_{i_0i_0}M_0+\sum_{j\neq i_0}^{n}c_{i_0j}u_j\geq 0.$$
Suppose $M_0\geq 0$. By the definition of $M_0$ and the assumptions (a)(b), we obtain
$$\Delta_{g_{i_0}} (u_{i_0}-M_0)+<X_{i_0},\nabla_{g_{i_0}} (u_{i_0}-M_0)>+c_{i_0i_0}(u_{i_0}-M_0)\geq 0.$$
From the assumptions (a)(b), $c_{i_0i_0}\leq 0$. Then from the strong maximum principle for single equations, we have $u_{i_0}\equiv M_0$, since the maximum of $u_{i_0}-M_0$ is $0$ and achieved by the interior point $x_0$. Since $u_{i_0}\leq 0$ on the boundary, we have $M_0\leq 0$, then $M_0=0.$
    Then we have
    $\sum\limits_{j\neq i_0}c_{i_0j}u_j\geq 0.$
    Since $u_j\leq M_0=0$, from the assumption (a), we have $c_{i_0j}u_j=0$, $j\neq i_0.$ For $j\in \beta$, $u_j<0$ in the interior. Hence for $j\in \beta$,  $c_{i_0j}\equiv 0$. 
    From the discussion above, we obtain $u_i\equiv M_i=0$ for  $i\in \alpha$ and $c_{ij}\equiv 0$ for $i\in \alpha,j\in \beta$. We finish the proof of Claim 1.

    Claim 2: Under the assumptions (a)(b)(c), if $M_0\geq 0$, then $u_i\equiv 0$ for every $1\leq i\leq n$.

    If $M_0<0$, then by the definition of $M_0,$ all $u_i<0$. Otherwise from Claim 1, the index set $\{1,\cdots,n\}$ is split up in two disjoint subsets $\alpha,\beta$ such that $c_{ij}\equiv 0$ for $i\in\alpha,j\in \beta.$ It contradicts to the assumption (c) unless $\beta$ is empty. If $\beta$ is empty, then $u_i\equiv 0$ for every $1\leq i\leq n$. We finish the proof of Claim 2.


    Now we remove the assumption $\max\limits_{x\in \partial M}u_i\leq 0$. Consider $$\tilde{u}_i=u_i-\max\limits_{x\in \partial M} \max\{u_i(x),0\}.$$
    Then $\max\limits_{x\in \partial M}\tilde{u}_i\leq 0.$ And $\tilde{f}_i=f_i-\max\limits_{x\in \partial M}\max\{u_i(x),0\}\cdot\sum\limits_{j=1}^nc_{ij}\geq 0.$ From Claim 1 above, we have $\tilde{u}_i\leq 0$, which implies
$$\max_{x\in \bar{M}} \max\{u_i(x),0\}\leq \max_{x\in \partial M} \max\{u_i(x),0\},\quad 1\leq i\leq n.$$
    
    Now we consider the strong maximum principle. If for some $i_0$, $u_{i_0}$ attains its non-negative maximum at an interior point, then $\tilde{u}_{i_0}$ attains its maximum $0$ at this point. From Claim 2 above, $u_i\equiv M_i$ for every $1\leq i\leq n$. Let $M_0=\max\limits_{i}M_i$. If $M_0=0$, then all $u_i$'s are identically $0$. In this case $f_i=0$, $1\leq i\leq n$, which contradicts to the assumption (f). Now suppose $M_0>0$. Let $\alpha_1$ be the subset of $\{1,\cdots,n\}$ such that $i\in\alpha_1$ if and only if $M_i=M_0$, and $\beta_1$ be the complement of $\alpha_1$. Then for  $i_0\in \alpha_1$, 
    \begin{eqnarray*}
0\leq f_{i_0}=\sum\limits_{j=1}^nc_{i_0j}u_j&=&c_{i_0i_0}M_{i_0}+\sum\limits_{j\neq i_0}c_{i_0j}M_j\\
&=&\sum\limits_{j=1}^nc_{i_0j}M_{0}+\sum\limits_{j\neq i_0}c_{i_0j}(M_j-M_{0})\leq 0.
    \end{eqnarray*}
    Therefore $\sum\limits_{j=1}^nc_{i_0j}=0$ and $c_{i_0j}(M_j-M_{0})=0$, $j\neq i_0$. And $f_{i_0}=0$. Then $c_{ij}\equiv 0$ for $i\in\alpha_1,j\in\beta_1$, which contradicts to the assumption (c) unless $\beta_1$ is empty. Thus under assumption (a)(b)(c), if for some $i_0$, $u_{i_0}$ attains its non-negative maximum at an interior point, then $u_i\equiv M_0, 1\leq i\leq n.$ And note that $M_0\geq 0$. 
In this case, $f_i=0$, $1\leq i\leq n$, which contradicts with the assumption (f).

If $M_0>0$, then $\sum\limits_{j=1}^nc_{ij}=0$ for $1\leq i\leq n,$ which contradicts with the assumption (e). Under the assumption (d), by Claim 1, we have $M_0\leq 0$. Thus $M_0=0$. 
\end{proof}
\begin{remark}\label{assumption (c)}
The assumption $(c)$ is easy to check by the following procedure. If $1\in\alpha$, consider $\beta_1=\{j: c_{1j}\equiv 0\}$, $\alpha_1=\{1,\cdots,n\}\setminus \beta_1$. Then $\alpha_1\cap \beta=\emptyset$. Then $\alpha_1\subseteq \alpha$. Denote $\alpha_0=\{1\}$. If $\alpha_{1}\subseteq\alpha_0$, then $\alpha=\alpha_0$ gives such a partition. If $\alpha_{1}\nsubseteq\alpha_0$, consider $\beta_2=\{j: c_{ij}\equiv 0, i\in \alpha_{0}\cup\alpha_1 \}$, $\alpha_2=\{1,\cdots,n\}\setminus \beta_2$. Then $\alpha_2\subseteq \alpha$. If $\alpha_2\subseteq \alpha_0\cup\alpha_1$, then $\alpha=\alpha_0\cup\alpha_1$ gives such a partition. If $\alpha_{2}\nsubseteq\alpha_0\cup\alpha_1$, consider $\beta_3=\{j: c_{ij}\equiv 0, i\in \bigcup_{k=0}^{2}\alpha_k \}$, $\alpha_3=\{1,\cdots,n\}\setminus \beta_3$. Repeat this procedure, then either we obtain a partition $\alpha,\beta$ such that $c_{ij}\equiv 0$ for $i\in\alpha,j\in\beta$ or we show that $1\notin \alpha$. If $1\notin \alpha$, repeat the procedure above for $2,3,\cdots, n$. Then we can show whether such a partition exists or not.
\end{remark}

\bibliographystyle{amsalpha}
\bibliography{bib}

\end{document}